\documentclass[12pt]{amsart}
\newtheorem{thm}{Theorem}[section]
\newtheorem{lem}[thm]{Lemma}
\newtheorem{cor}[thm]{Corollary}
\newtheorem{prop}[thm]{Proposition}

\theoremstyle{remark}
\newtheorem{defn}[thm]{Definition}
\newtheorem{rem}[thm]{Remark}
\newtheorem{exa}[thm]{Example}

\newtheorem*{prfofthm1}{Proof of Theorem~\ref{lowerboundthm}}
\newtheorem*{prfofthm2}{Proof of Theorem~\ref{cor}}
\newtheorem*{acknowledgement}{Acknowledgment}
\makeatletter
 
 \@addtoreset{equation}{section}

\makeatother

\title{Slope equality of plane curve fibrations and its application to Durfee's conjecture}
\author{Makoto Enokizono}
\address{Makoto Enokizono,
	Department of Mathematics,
	Graduate School of Science,
	Osaka University,
	Toyonaka, Osaka 560-0043, Japan}
\email{m-enokizono@cr.math.sci.osaka-u.ac.jp}
\usepackage[top=35truemm,bottom=35truemm,left=25truemm,right=25truemm]{geometry}
\usepackage{amsmath,cases}
\usepackage{amsfonts}
\usepackage[all]{xy}
\usepackage{here}
\usepackage{latexsym}
 \def\qed{\hfill $\Box$} 
\begin{document}
\maketitle

\begin{abstract}
We give a slope equality for fibered surfaces whose general fiber is a smooth plane curve.
As a corollary, we prove a ``strong'' Durfee-type inequality for isolated hypersurface surface singularities, which implies Durfee's strong conjecture for such singularities with non-negative topological Euler number of the exceptional set of the minimal resolution.
\end{abstract}

\section*{Introduction}
Throughout this paper, we work over the complex number field $\mathbb{C}$.
Let $f\colon S\to B$ be a fibered surface of genus $g$, that is, a surjective morphism from a non-singular projective surface $S$ to a non-singular projective curve $B$ whose general fiber $F$ is a non-singular curve of genus $g$.
Let $K_f=K_S-f^{*}K_B$ denote the relative canonical bundle of $f$ and put $\chi_f:=\mathrm{deg}f_*\mathcal{O}(K_f)$.
The ratio $K_f^2/\chi_f$ of the self-intersection number $K_f^2$ and $\chi_f$ is called the {\em slope of $f$}.

In this paper, we consider fibered surfaces whose general fiber is a plane curve of degree $d$ which are called {\em plane curve fibrations of degree $d$}.
A plane curve fibration of degree $1$ or $2$ is a ruled surface
and that of degree $3$ is nothing but an elliptic surface.
In the sequel, we always assume that $d$ is greater than $3$.
Note that a plane curve fibration of degree $4$ is nothing but a non-hyperelliptic fibration of genus $3$.
Let $\mathcal{A}_d$ be the set of holomorphically equivalence classes of fiber germs whose general fiber is a smooth plane curve of degree $d$ (see \S4).
Then our main theorem is as follows.

\begin{thm}\label{Intromainthm}
Let $d\ge 4$ be an integer. 
Then there exists a non-negative function $\mathrm{Ind}_d\colon \mathcal{A}_d\to \frac{1}{d-2}\mathbb{Z}_{\ge 0}$ such that
for any relatively minimal plane curve fibration $f\colon S\to B$ of degree $d$, the value $\mathrm{Ind}_d(F)$ equals to $0$ for any general fiber $F$ of $f$ and 
\begin{equation}\label{Introslopeeq}
K_f^2=\frac{6(d-3)}{d-2}\chi_f+\sum_{p\in B}\mathrm{Ind}_d(F_p)
\end{equation}
holds, where $F_p:=f^{-1}(p)$ denotes the fiber germ over $p\in B$.
\end{thm}

The value $\mathrm{Ind}_d(F_p)$ is nowadays called a {\em Horikawa index of $F_p$} and
the equality \eqref{Introslopeeq} a {\em slope equality} for plane curve fibrations of degree $d$ (cf.\ \cite{ak}).
In the case of $d=4$, that is, non-hyperelliptic fibrations of genus $3$, Theorem~\ref{Intromainthm} was first obtained by Reid \cite{Re} which was generalized for fibered surfaces of odd genus $g$ whose general fiber has maximal Clifford index by Konno \cite{Kon}.
The lower bound of the slope of plane curve fibrations of degree $5$ was obtained by Barja-Stoppino~\cite{BaSt}.


Before stating an application of Theorem~\ref{Intromainthm}, let us explain the background of Durfee's conjecture.
Let $(X,0)$ be an isolated hypersurface surface singularity, that is, $X=\{h(x,y,z)=0\}\subset \mathbb{C}^3$ for some analytic function $h$ on a neighborhood at the origin $0\in \mathbb{C}^3$ with an isolated singularity $0\in X$.
The geometric genus $p_g$ of $(X,0)$ is defined by $\mathrm{dim}H^{1}(\mathcal{O}_{\widetilde{X}})$, where $\widetilde{X}\to X$ is a resolution.
Let $M=X_{\varepsilon}\cap B$ be a (generic) Milnor fiber, where $X_{\varepsilon}=\{h(x,y,z)=\varepsilon\}$ is a smoothing of $(X,0)$ and $B\subset \mathbb{C}^3$ is a small closed ball centered at the origin.
The rank $\mu$ of the second homology group $H_2(M,\mathbb{Z})$ is called the Milnor number.
Let $\mu_{+}$ (resp.\ $\mu_{-}$, $\mu_0$) be the number of positive (resp.\ negative, 0) eigenvalues of the natural intersection form $H_2(M,\mathbb{Z})\times H_2(M,\mathbb{Z})\to \mathbb{Z}$.
Then $\mu=\mu_{+}+\mu_{-}+\mu_0$ and $\sigma=\mu_{+}-\mu_{-}$ is called the signature.
The original Durfee's conjectures \cite{Dur} for hypersurface singularities are as follows:

\smallskip

\noindent
(Weak conjecture) $\sigma\le 0$.

\smallskip

\noindent
(Strong conjecture) $6p_g\le \mu$.

\smallskip

\noindent
From Durfee's result $2p_g=\mu_{+}+\mu_0$ \cite{Dur}, the weak conjecture is equivalent to $4p_g\le \mu+\mu_0$.
Thus the strong conjecture implies the weak conjecture.
Koll\'ar and N\'emethi showed in \cite{KoNe} that the weak conjecture is true.
Moreover, they showed that the strong conjecture is true for hypersurface singularities with integral homology sphere link.
As a remarkable application of Theorem~\ref{Intromainthm}, we prove the following Durfee-type inequality for $2$-dimentional isolated hypersurface singularities, which implies that the strong conjecture is true for a large class of hypersurface singularities:

\begin{thm} \label{Introcor}
Let $(X,0)$ be an isolated hypersurface surface singularity with Milnor number $\mu$ and geometric genus $p_g>0$.
Then we have
$$
6p_g\le \mu-\chi_{\mathrm{top}}(A),
$$
or equivalently,
$$
\sigma\le -2p_g-1-s,
$$
where $\chi_{\mathrm{top}}(A)$ is the topological Euler number of the exceptional set $A$ of the minimal resolution $\pi\colon \widetilde{X}\to X$
and $s$ is the number of irreducible components of $A$.
In particular, the strong conjecture holds if $\chi_{\mathrm{top}}(A)\ge 0$
and the week conjecture holds for any isolated hypersurface surface singularity.
\end{thm}


The strategy of the proof of Theorem~\ref{Intromainthm} is as follows.
Put $\lambda_d:=6(d-3)/(d-2)$.
Given a plane curve fibration $f\colon S\to B$ of degree $d$,
we will show that there is a line bundle $\mathcal{L}$ on $S$ such that the restriction $\mathcal{L}|_{F}$ to the general fiber $F$ defines the embedding $F\subset \mathbb{P}^2$ in \S1.
Using the line bundle $\mathcal{L}$, we will show in \S2 that the difference $K_f^2-\lambda_d\chi_f$ can be localized on a finite number of fiber germs, that is, we can define $\mathrm{Ind}_d(F_p)$ for any fiber germ $F_p$ of $f$.
But the non-negativity of $\mathrm{Ind}_d(F_p)$ seems not to follow directly from the definition,
because it contains both positive and negative terms.
Thus we will show firstly a slope inequality $K_f^2-\lambda_d\chi_f\ge 0$ in \S3.
The essential idea of the proof is to apply the Hilbert stability of the Veronese surfaces (cf.\ \cite{Ke}) to the result of Barja-Stoppino \cite{BaSt2}.
In order to deduce the non-negativity of the Horikawa index from the slope inequality, we will use an algebraization of any fiber germ in $\mathcal{A}_d$ in \S4. 
Roughly speaking, for an arbitrary fiber germ $F_0$ in $\mathcal{A}_d$, we construct a global plane curve fibration $\overline{f}\colon \overline{S}\to \mathbb{P}^{1}$ of degree $d$ whose central fiber $\overline{F}_0=\overline{f}^{-1}(0)$ is an ``approximation'' of $F_0$ and any other singular fiber is an irreducible Lefschetz plane curve.
Since we can show that $\mathrm{Ind}_d(F'_0)=0$ for any irreducible Lefschetz fiber germ $F'_0$, we in particular have $\mathrm{Ind}_d(F_0)=\mathrm{Ind}_d(\overline{F}_0)=K_{\overline{f}}^2-\lambda_d\chi_{\overline{f}}$.
Thus the slope inequality $K_{\overline{f}}^2-\lambda_d\chi_{\overline{f}}\ge 0$ implies the non-negativity of $\mathrm{Ind}_d(F_0)$ for any fiber germ $F_0$ in $\mathcal{A}_d$.

In \S5, we will discuss the signature of surfaces with plane curve fibrations.
We can define a local signature for plane curve fibrations by using the Horikawa index in Theorem~\ref{Intromainthm} (cf.\ \cite{ak}).
On the other hand, Kuno~\cite{Ku} defined another local signature for these fibrations by using Meyer's signature cocycle from the topological point of view.
We will show the coincidence of the two local signatures similarly as in \cite{Te}.

In \S6, we prove Theorem~\ref{Introcor}.
The essential point of the proof is that the minimal resolution space of any $2$-dimentional hypersurface singularity can be embedded in a relatively minimal plane curve fibration of high degree and the Horikawa index of the fiber germ containing the exceptional curves can be described by some invariants of singularities.
The non-negativity of the Horikawa index implies Theorem~\ref{Introcor}.


\begin{acknowledgement}
I would like to express special thanks to Prof.\ Kazuhiro Konno for a lot of discussions and supports.
I also thank Prof.\ Tadashi Ashikaga for useful comments in \S4 and \S5 and discussions on Durfee's conjecture.
The research is supported by JSPS KAKENHI No.\ 16J00889.
\end{acknowledgement}


\section{Glueing linear series}

For a smooth projective curve $C$ (resp.\ a family of smooth projective curves $f\colon X\to B$), let $\mathcal{G}^{r}_{d}(C)$ (resp.\ $\mathcal{G}^{r}_{d}(f)$) be the (resp.\ relative) Brill-Noether variety parametrizing $\mathfrak{g}^{r}_{d}$'s on $C$ (resp.\ on fibers of $f$), where we denote by $\mathfrak{g}^r_d$ a linear system of degree $d$ and of dimension $r$ (cf.\ \cite{ACG2} Chapter~XXI).

In this section, we prove the following theorem for the later use, which is a slight improvement of Theorem~3.1 in \cite{barja-naranjo}.

\begin{thm} \label{glueingthm}
Let $X$, $B$ be normal algebraic varieties (resp.\ normal analytic varieties)
and $f\colon X\to B$ a proper flat morphism whose general fiber is a non-singular projective curve.
Let $B_0\subset B$ be the Zariski open subset consisting of smooth points $p$ of $B$ such that $F_p=f^{-1}(p)$ is non-singular
and $f_0\colon X_0=f^{-1}(B_0)\to B_0$ the restriction of $f$ to $B_0$.
Let $r,d$ be positive integers.
Assume that there exists a rational section $\eta\colon B_0\dasharrow \mathcal{G}^{r}_{d}(f_0)$.
Then there exist a divisorial sheaf $\mathcal{L}$ on $X$ and a subsheaf $\mathcal{G}\subset f_{*}\mathcal{L}$ such that the linear subspace $\mathcal{G}\otimes \mathbb{C}(p)\subset H^{0}(F_p, \mathcal{L}|_{F_p})$ defines $\eta(p)$ for any general $p\in B_0$.
\end{thm}

\begin{proof}
We may assume that $\eta(p)$ is base point free for any general $p\in B_0$ by removing the locus of all base points of $\eta(p)$, $p\in B_0$.
Shrinking $B_0$ if necessary, we may assume that $\eta$ is a section.
For $p\in B_0$, we can write $\eta(p)=\{D(p)_{\lambda}\}_{\lambda\in \mathbb{P}^{r}}$, where $D(p)_{\lambda}$ is an effective divisor of degree $d$ on $F_p$.
Let $\mathcal{E}_0$ be a locally free sheaf on $B_0$ such that $X_0$ is embedded in $\mathbb{P}_{B_0}(\mathcal{E}_0)$ over $B_0$ (such $\mathcal{E}_0$ exists, e.g., take the direct image sheaf of a sufficiently $f_0$-ample invertible sheaf on $X_0$).
We regard each fiber $F_p$ as a subvariety of $\mathbb{P}(\mathcal{E}_0\otimes \mathbb{C}(p))\simeq \mathbb{P}^{\mathrm{rank}(\mathcal{E}_0)-1}$ via the inclusion $X_0\subset \mathbb{P}_{B_0}(\mathcal{E}_0)$.
Let $\overline{D(p)}_{\lambda}$ denote the plane in $\mathbb{P}(\mathcal{E}_0\otimes \mathbb{C}(p))$ spanned by $D(p)_{\lambda}$.
Then the dimension $k:=\mathrm{dim}\overline{D(p)}_{\lambda}$ does not depend on the choices of $p$ and $\lambda$ from the Riemann-Roch theorem.
Now, we consider the subvariety $P$ of the relative Grassmannian $Gr_{B_0}(k,\mathbb{P}(\mathcal{E}_0))=\cup_{p\in B_0}Gr(k,\mathbb{P}(\mathcal{E}_0\otimes \mathbb{C}(p)))$ defined by
$$
P:=\{[\overline{D(p)}_{\lambda}]\in Gr(k,\mathbb{P}(\mathcal{E}_0\otimes \mathbb{C}(p)))|\lambda\in \mathbb{P}^{r}, p\in B_0\}.
$$
It is a holomorphic $\mathbb{P}^{r}$-bundle over $B_0$ via the natural projection.
We can define a morphism $\Phi$ from $X_0$ to $P^{*}:=Gr_{B_0}(r-1,P)$ by mapping $x$ to $\{[\overline{D(p)}_{\lambda}]|x\in D(p)_{\lambda}\}$, the restriction of which to the fiber $F_p$ is nothing but the morphism associated with $\eta(p)$.
Let $\mathcal{G}_0$, $\mathcal{L}_0$ respectively be the direct image sheaf of the tautological line bundle $\mathcal{O}_{P^{*}}(1)$ via the natural projection $P^{*}\to B_0$, the pull-back of $\mathcal{O}_{P^{*}}(1)$ via $\Phi$.
It follows that $P^{*}=\mathcal{P}_{B_0}(\mathcal{G}_0)$ and $\mathcal{G}_0\subset f_{0*}\mathcal{L}_0$.
Let $i_{B_0}\colon B_0\to B$ and $i_{X_0}\colon X_0\to X$ be the natural inclusions.
We put $\mathcal{G}:=i_{B_0*}\mathcal{G}_0$ and $\mathcal{L}:=(i_{X_0*}\mathcal{L})^{**}$, which are the desired sheaves. 
Indeed, we have $\mathcal{G}\subset i_{B_0*}f_{0*}\mathcal{L}_0=f_{*}i_{X_0*}\mathcal{L}_0\subset f_{*}\mathcal{L}$.
\end{proof}

\begin{rem}
If $f\colon S\to B$ has a section, Theorem~\ref{glueingthm} follows directly from the existence of the relative Poincar\'e line bundle (cf.\ \cite{ACG2}).
\end{rem}

\begin{cor}
Let $X$ and $B$ be normal algebraic varieties and $f\colon X\to B$ and $d,r$ as in Theorem~\ref{glueingthm}.
Assume that the fiber $F_p=f^{-1}(p)$ has a base point free $\mathfrak{g}^{r}_{d}$ for general $p\in B$.
Then, after a suitable finite base change $B'\to B$, there exist a $\mathbb{P}^{r}$-bundle $P'$ over $B'$ and a rational map $\varphi\colon X'\dasharrow P'$ over $B'$ of degree $d$, where $f'\colon X'\to B'$ is a base change fibration of $f$.
\end{cor}

\begin{proof}
By assumption, the general fiber of $\mathcal{G}^{r}_{d}(f_0)\to B_0$ is non-empty.
Since $\mathcal{G}^{r}_{d}(f_0)$ is algebraic, we can take a subvariety $B'_0$ of $\mathcal{G}^{r}_{d}(f_0)$ such that the natural map $B'_0\to B_0$ is finite (after shrinking $B_0$ if necessary).
We take a compactification $B'\to B$ of it and perform base change via this map.
Let $f'\colon S'\to B'$ be the base change fibration of $f$ and $f'_0\colon X'_0\to B'_0$ the restriction of $f'$ to $X'_0=f'^{-1}(B'_0)$.
Since $\mathcal{G}^{r}_{d}(f'_0)=\mathcal{G}^{r}_{d}(f_0)\times_{B_0}B'_0$, we can take a section $B'_0\to \mathcal{G}^{r}_{d}(f'_0)$ by $p\mapsto (p,p)$.
From Theorem~\ref{glueingthm}, there exist a line bundle $\mathcal{L}$ on $B'$ and a subbundle $\mathcal{G}\subset f'_{*}\mathcal{L}$ such that the rational map $X'\dasharrow \mathbb{P}_{B'}(\mathcal{G})$ associated to $f'^{*}\mathcal{G}\to \mathcal{L}$ is of degree $d$.

\end{proof}

\section{Localizations}

Let $f\colon S\to B$ be a fibered surface of genus $g=(d-1)(d-2)/2>1$ whose general fiber has a very ample $\mathfrak{g}^{2}_{d}$, that is, it is a smooth plane curve of degree $d$.
Since the $\mathfrak{g}^{2}_{d}$ is unique, there exists a line bundle $\mathcal{L}$ on $S$ (unique up to a multiple of a divisor consisting of components of fibers) such that $\mathcal{L}|_{F}$ is the $\mathfrak{g}^{2}_{d}$ on any general fiber $F$ by Theorem~\ref{glueingthm}.
Then $\omega_f=\mathcal{O}_S(K_f)$ is isomorphic to $\mathcal{L}^{\otimes d-3}(J)$ for some divisor $J$ consisting of components of fibers (it depends on the choice of $\mathcal{L}$)
since $\mathcal{L}^{\otimes d-3}|_F$ is the canonical bundle $K_F$ for a general fiber $F$.
On the other hand, for $k=1,\ldots,d-1$, there exists a natural exact sequence
$$
0\to \mathrm{Sym}^{k}f_{*}\mathcal{L}\to f_{*}\mathcal{L}^{\otimes k}\to \mathcal{T}_k\to 0
$$
induced from the multiplicative map $\mathrm{Sym}^{k}H^{0}(\mathcal{L}|_{F})\to H^{0}(\mathcal{L}^{\otimes k}|_{F})$ on fibers,
where the cokernel $\mathcal{T}_k$ is a torsion sheaf.
Thus, we get

\begin{equation} \label{k-resteq}
\mathrm{deg}(f_{*}\mathcal{L}^{\otimes k})=\mathrm{deg}(\mathrm{Sym}^{k}f_{*}\mathcal{L})+\mathrm{length}(\mathcal{T}_k).
\end{equation}
By the Grothendieck Riemann-Roch theorem, we have
\begin{equation} \label{k-GRReq}
\mathrm{deg}(f_{*}\mathcal{L}^{\otimes k})-\mathrm{deg}(R^{1}f_{*}\mathcal{L}^{\otimes k})=\frac{k^2}{2}L^{2}-\frac{k}{2}LK_f+\chi_f,
\end{equation}
where $L=c_1(\mathcal{L})$.
From \eqref{k-resteq} and \eqref{k-GRReq}, we obtain
\begin{align} \label{keq}
&\frac{k^2}{2}L^{2}-\frac{k}{2}LK_f+\chi_f+\mathrm{deg}(R^{1}f_{*}\mathcal{L}^{\otimes k})-\mathrm{length}(\mathcal{T}_k)\\
&=\binom{k+2}{3}\left(\frac{1}{2}L^{2}-\frac{1}{2}LK_f+\chi_f+\mathrm{deg}(R^{1}f_{*}\mathcal{L})\right). \nonumber
\end{align}
Note that two sheaves $R^{1}f_{*}\mathcal{L}^{\otimes d-2}$ and $R^{1}f_{*}\mathcal{L}^{\otimes d-1}$ are torsion sheaves.
Since $\binom{d+1}{3}$ times the left hand side of \eqref{keq} for $k=d-2$ is equal to $\binom{d}{3}$ times the left hand side of \eqref{keq} for $k=d-1$, we obtain by a calculation and $(d-3)L= K_f-J$ that
\begin{equation} \label{slopeeq}
K_f^{2}=\frac{6(d-3)}{d-2}\chi_f+\sum_{p\in B}\mathrm{Ind}_{d}(F_p),
\end{equation}
where $J=\sum_{p\in B}J_p$ is the natural decomposition such that any component of $J_p$ is contained in $F_p$ and we put
\begin{align*}
\mathrm{Ind}_{d}(F_p):=&J_p^{2}+2(d-3)\left(\frac{d+1}{d-2}\mathrm{length}_p(R^{1}f_{*}\mathcal{L}^{\otimes d-2})-\mathrm{length}_p(R^{1}f_{*}\mathcal{L}^{\otimes d-1})\right)\\
&+2(d-3)\left(\mathrm{length}_p(\mathcal{T}_{d-1})-\frac{d+1}{d-2}\mathrm{length}_p(\mathcal{T}_{d-2})\right).
\end{align*}
It follows from \eqref{slopeeq} that the value $\mathrm{Ind}_d(F_p)$ is independent of the choice of the line bundle $\mathcal{L}$ since $\mathcal{L}$ is unique up to a multiple of an $f$-vertical divisor.
Indeed, for any divisor $I_p$ consisting of components of $F_p$, the values of $\mathrm{Ind}_d(F_q)$ defined by $\mathcal{L}$ and $\mathcal{L}(I_p)$ are the same for any $q\neq p$.
Thus it also holds for $q=p$ from \eqref{slopeeq}.

But the non-negativity of $\mathrm{Ind}_d(F_p)$ seems not to follow directly from the definition,
because it contains both positive and negative terms.

\section{Lower bound of the slope}

In this section, we prove the following inequality for plane curve fibrations.

\begin{thm} \label{lowerboundthm}
Let $f\colon S\to B$ be a relatively minimal plane curve fibration of degree $d\ge 4$.
Then we have
$$
K_f^{2}\ge \frac{6(d-3)}{d-2}\chi_f.
$$
\end{thm}

Let $f\colon S\to B$ be a relatively minimal plane curve fibration of degree $d$.
Since the $\mathfrak{g}^{2}_{d}$ on the general fiber $F$ is unique, there exists a line bundle $\mathcal{L}$ on $S$ such that the restriction $\mathcal{L}|_{F}$ is the $\mathfrak{g}^{2}_{d}$ and it is unique up to a multiple of divisors consisting of components of fibers.
Since $\mathcal{L}|_{F}^{\otimes d-3}=\omega_{F}$, we can write $\mathcal{L}^{\otimes d-3}(J)=\omega_f$ for some divisor $J$ consisting of components of fibers. 
Tensoring components of fibers to $\mathcal{L}$, we may assume that $J$ is effective.
Then we have an injection $f_*\mathcal{L}^{\otimes d-3}\to f_*\omega_f$.
The composite of it and the natural homomorphism $\mathrm{Sym}^{d-3}f_*\mathcal{L}\to f_*\mathcal{L}^{\otimes d-3}$ induces an injection $\mathrm{Sym}^{d-3}f_*\mathcal{L}\to f_*\omega_f$ whose cokernel is a torsion sheaf.
Let $\mathfrak{c}$ be the maximal effective divisor on $B$ such that the image of the homomorphism $\mathrm{Sym}^{d-3}f_*\mathcal{L}\to f_*\omega_f$ is contained in $f_*\omega_f(-\mathfrak{c})$.
Then there is an exact sequence
$$
0\to \mathrm{Sym}^{d-3}f_*\mathcal{L}\to f_*\omega_f(-\mathfrak{c})\to \mathcal{T}\to 0,
$$
which induces an elementary transformation 
$$
P:=\mathbb{P}_B(f_*\omega_f)=\mathbb{P}_B(f_*\omega_f(-\mathfrak{c}))\xleftarrow{\tau} \widetilde{P}\xrightarrow{\tau'} P':=\mathbb{P}_B(\mathrm{Sym}^{d-3}(f_*\mathcal{L}))
$$
such that 
$$
\tau^{*}\mathcal{O}_{\mathbb{P}_B(f_*\omega_f(-\mathfrak{c}))}(1)-E_{\tau}=\tau'^{*}\mathcal{O}_{\mathbb{P}_B(\mathrm{Sym}^{d-3}(f_*\mathcal{L}))}(1)
$$ 
holds, where $\mathcal{O}_{\mathbb{P}_B(\mathcal{E})}(1)$ is the tautological line bundle associated with $\mathcal{E}$ and $E_{\tau}$ is an effective exceptional divisor of $\tau$.
On the other hand, we have
$$
\mathcal{O}_{\mathbb{P}_B(f_*\omega_f(-\mathfrak{c}))}(1)=\mathcal{O}_{\mathbb{P}_B(f_*\omega_f)}(1)-\pi^{*}\mathfrak{c},
$$
and then we get
$$
\tau^{*}\mathcal{O}_{\mathbb{P}_B(f_*\omega_f)}(1)-\widetilde{\pi}^{*}\mathfrak{c}
-E_{\tau}=\tau'^{*}\mathcal{O}_{\mathbb{P}_B(\mathrm{Sym}^{d-3}(f_*\mathcal{L}))}(1),
$$
where $\pi\colon P\to B$, $\widetilde{\pi}\colon \widetilde{P}\to B$ are the natural projections.
Now we consider the relative Veronese embedding $W':=\mathbb{P}_B(f_*\mathcal{L})\to P'$ of degree $d-3$ corresponding to the surjective homomorphism
$\phi'^{*}\mathrm{Sym}^{d-3}(f_*\mathcal{L})\to \mathcal{O}_{\mathbb{P}_B(f_*\mathcal{L})}(d-3)$, where $\phi'\colon W'\to B$ is the natural projection.
There is a rational map $S\dasharrow W'$ corresponding to the homomorphism $f^*f_*\mathcal{L}\to \mathcal{L}$.
Let $X'\subset W'$ be (the closure of) its image.
Let $\widetilde{W}$, $\widetilde{X}$ be the proper transforms of $W'$, $X'$ with respect to $\tau'$ and $W$, $X$ the image of $\widetilde{W}$, $\widetilde{X}$ via $\tau$, respectively.
Note that $X$ coincides with the image of the relative canonical map $S\dasharrow P$ and two birational maps $S\dasharrow X \dasharrow \widetilde{X}$ and $S\dasharrow X'\dasharrow \widetilde{X}$ coincide.
Let $\rho\colon \widetilde{S}\to S$ be the resolution of indeterminacy of $S\dasharrow \widetilde{X}$ and $\widetilde{\Phi}\colon \widetilde{S}\to \widetilde{X}$ the induced birational morphism.
We put $T:=\mathcal{O}_{\mathbb{P}_B(f_*\omega_f)}(1)$, $T':=\mathcal{O}_{\mathbb{P}_B(\mathrm{Sym}^{d-3}(f_*\mathcal{L}))}(1)$ and denote also $\tau^{*}T$, $\tau'^{*}T'$ by $T$, $T'$ for simplicity.
Let $\Gamma$, $\Gamma'$ respectively be the numerical equivalence classes of fibers of $\pi\colon P\to B$, $\pi'\colon P'\to B$.
Note that $\tau^*\Gamma=\tau'^*\Gamma'$ and we also denote it by $\Gamma$ or $\Gamma'$.
From the above arguments, we have
$$
T-T'\equiv c\Gamma+E_{\tau},
$$
where $c$ is the degree of $\mathfrak{c}$ and the symbol $\equiv$ means the numerical equivalence.
Put $N:=T|_{\widetilde{W}}$, $N':=T'|_{\widetilde{W}}$, $M:=\widetilde{\Phi}^{*}T$ and $M':=\widetilde{\Phi}^{*}T'$.
The numerical equivalence classes of $W'$, $X'$ in $P'$ as cycles can be written by
$$
W'\equiv (d-3)^2T'^{g-3}+\alpha' T'^{g-4}\Gamma',\quad X'\equiv d(d-3)T'^{g-2}+\beta' T'^{g-3}\Gamma'
$$
for some $\alpha'$, $\beta'$.
Then we have
$$
N'^3=T'^3W'=(d-3)^2(\chi_f-l)+\alpha',\quad M'^2=T'^2X'=d(d-3)(\chi_f-l)+\beta',
$$
where $l:=\mathrm{length}(f_*\omega_f/\mathrm{Sym}^{d-3}f_*\mathcal{L})\ge 0$.
Note that $T'|_{W'}=\mathcal{O}_{\mathbb{P}_B(f_*\mathcal{L})}(d-3)$ and then we have 
$$
N'^3=(d-3)^3\mathrm{deg}f_*\mathcal{L}.
$$
Then the numerical class of the canonical divisor $K_{W'}$ of $W'$ is 
\begin{align*}
K_{W'}&\equiv -3\mathcal{O}_{\mathbb{P}_B(f_*\mathcal{L})}(1)+(\mathrm{deg}f_*\mathcal{L}+2b-2)\Gamma'|_{W'} \\
&=-3\mathcal{O}_{\mathbb{P}_B(f_*\mathcal{L})}(1)+\left(\frac{N'^3}{(d-3)^3}+2b-2\right)\Gamma'|_{W'}, 
\end{align*}
where $b:=g(B)$, the genus of $B$.
The numerical class $[X']_{W'}$ of $X'$ in $W'$ can be denoted by
$$
[X']_{W'}\equiv d\mathcal{O}_{\mathbb{P}_B(f_*\mathcal{L})}(1)+\beta''\Gamma'|_{W'}
$$
for some $\beta''$.
Since $\mathcal{O}_{\mathbb{P}_B(f_*\mathcal{L})}(d-3)=T'|_{W'}$, we have
$$
\left(\frac{d}{d-3}T'+\beta''\Gamma\right)W'=X'
$$
and thus we get
$$
\beta'=(d-3)^2\beta''+\frac{d}{d-3}\alpha'.
$$

By the definition of $M$, we can write $\rho^{*}K_f=M+Z$ for some effective vertical divisor $Z$ with respect to $\widetilde{f}\colon \widetilde{S}\to B$.
Then we have
\begin{equation} \label{KfMeq}
K_f^2=(\rho^*K_f)^2=(M+Z)^2=M^2+(\rho^{*}K_f+M)Z\ge M^2,
\end{equation}
where the last inequality follows from the nefness of $K_f$.

\begin{lem} \label{singlem}
$$
M^2\ge \frac{d-1}{d-3}N^3.
$$
\end{lem}

\begin{proof}
Take a sufficiently ample divisor $\mathfrak{a}$ such that $|M'+\widetilde{f}^*\mathfrak{a}|$ is free from base points.
Then we can take a smooth general member $C\in |M'+\widetilde{f}^*\mathfrak{a}|$ by Bertini's theorem.
Let $C':=(\tau'\circ\widetilde{\Phi})(C)$.
Now we compare the genus $g(C)$ of $C$ and the arithmetic genus $p_a(C')$ of $C'$.

First, we compute $g(C)$.
The adjunction formula says that

\begin{align}
2g(C)-2&=(K_{\widetilde{S}}+C)C \nonumber \\
&=(\rho^*K_f+E+(2b-2)\widetilde{F}+C)C \nonumber \\
&=(M+Z+E+(2b-2)\widetilde{F}+M'+a\widetilde{F})(M'+a\widetilde{F}) \nonumber \\
&=(2M+Z+E+(2b-2+a-c)\widetilde{F}-\widetilde{\Phi}^*E_{\tau})(M+(a-c)\widetilde{F}-\widetilde{\Phi}^*E_{\tau}) \nonumber \\
&=2M^2+(Z+E)(M-\widetilde{\Phi}^*E_{\tau})+(2b-2+3a-3c)(2g-2)+(\widetilde{\Phi}^*E_{\tau})^2, \label{Ceq}
\end{align}
where $b:=g(B)$, $\widetilde{F}$ is the numerical class of a fiber of $\widetilde{f}$, $E$ is the exceptional divisor of $\rho$ such that $K_{\widetilde{S}}=\rho^{*}K_S+E$ and $a:=\mathrm{deg}\mathfrak{a}$.

Next, we compute $p_a(C')$.
The adjunction formula also says that

\begin{align}
2p_a(C')-2&=(K_{X'}+C')C' \nonumber \\
&=((K_{W'}+[X']_{W'})|_{X'}+C')C' \nonumber \\
&=\left(\left(T'+\left(\frac{N'^3}{(d-3)^3}+2b-2+\beta''\right)\Gamma'\right)|_{X'}+(T'+a\Gamma')|_{X'}\right)(T'+a\Gamma')|_{X'} \nonumber \\
&=\left(2T'+\left(\frac{N'^3}{(d-3)^3}+2b-2+\beta''+a\right)\Gamma'\right)(T'+a\Gamma')X' \nonumber \\
&=\left(2T'^2+\left(\frac{N'^3}{(d-3)^3}+2b-2+\beta''+3a\right)T'\Gamma'\right)
(d(d-3)T'^{g-2}+\beta'T'^{g-3}\Gamma') \nonumber \\
&=2d(d-3)(\chi_f-l)+d(d-3)\left(\frac{N'^3}{(d-3)^3}+2b-2+\beta''+3a\right)+2\beta' \nonumber \\
&=-\frac{d(d-1)}{(d-3)^2}N'^3+\frac{3d-6}{d-3}M'^2+(2b-2+3a)(2g-2) \nonumber \\
&=-\frac{d(d-1)}{(d-3)^2}N^3+\frac{3d-6}{d-3}M^2+\frac{d(d-1)}{(d-3)^2}(E_{\tau}|_{\widetilde{W}})^{3}+\frac{3d-6}{d-3}(\widetilde{\Phi}^{*}E_{\tau})^2 \nonumber \\
&\ \ \ +(2b-2+3a-3c)(2g-2), \label{C'eq}
\end{align}
where the last equality follows from $N^3-(E_{\tau}|_{\widetilde{W}})^3=N'^3+3c(d-3)^2$ and $M^2+(\widetilde{\Phi}^{*}E_{\tau})^2=M'^2+2cd(d-3)$.
From \eqref{Ceq} and \eqref{C'eq}, we get

\begin{align}
2p_a(C')-2g(C)=&-\frac{d(d-1)}{(d-3)^2}N^3+\frac{d}{d-3}M^2-(Z+E)(M-\widetilde{\Phi}^*E_{\tau}) \nonumber \\
&+\frac{d(d-1)}{(d-3)^2}(E_{\tau}|_{\widetilde{W}})^{3}+\frac{2d-3}{d-3}(\widetilde{\Phi}^{*}E_{\tau})^2 \label{C'-Ceq}
\end{align}
and it is non-negative since $C\to C'$ is birational.
On the other hand, we have
\begin{align}
(E_{\tau}|_{\widetilde{W}})^3&=(N-N'-c\Gamma|_{\widetilde{W}})^2E_{\tau}|_{\widetilde{W}}=N'^2E_{\tau}|_{\widetilde{W}} \nonumber \\
&=T'^2E_{\tau}\widetilde{W}=T'^2E_{\tau}W'=(d-3)^2T'^{g-1}E_{\tau} \label{exWeq}
\end{align}
and
\begin{align}
(\widetilde{\Phi}^*E_{\tau})^2&=(M-M'-c\widetilde{F})\widetilde{\Phi}^*E_{\tau}=-M'\widetilde{\Phi}^*E_{\tau} \nonumber \\
&=-T'E_{\tau}\widetilde{X}=-T'E_{\tau}X'=-d(d-3)T'^{g-1}E_{\tau}. \label{exXeq}
\end{align}
Note that $T'^{g-1}E_{\tau}=\mathrm{length}\mathcal{T}\ge 0$ by a simple computation.
From \eqref{C'-Ceq}, \eqref{exWeq}, \eqref{exXeq} and $(Z+E)(M-\widetilde{\Phi}^*E_{\tau})=(Z+E)C\ge 0$, we have

\begin{align*}
-\frac{d(d-1)}{(d-3)^2}N^3+\frac{d}{d-3}M^2\ge (Z+E)C+d(d-2)\mathrm{length}\mathcal{T}\ge 0,
\end{align*}
which is the desired inequality.
\end{proof}

\begin{lem} \label{stablem}
$$
N^3\ge \frac{6(d-3)^2}{(d-1)(d-2)}\chi_f.
$$
\end{lem}

\begin{proof}
Since the linear system $\widetilde{\phi}_*N\otimes \mathbb{C}(p)=H^0(\widetilde{\phi}^{-1}(p),N|_{\widetilde{\phi}^{-1}(p)})$ on a general fiber $\widetilde{\phi}^{-1}(p)\simeq \mathbb{P}^2$ induces a Veronese embedding of degree $d-3$, the pair $(\widetilde{\phi}^{-1}(p), \widetilde{\phi}_*N\otimes \mathbb{C}(p))$ is Hilbert stable by Corollary~5.3 in \cite{Ke}.
Thus we can apply Theorem~6 in \cite{BaSt2} to the pair $(N,\widetilde{\phi}_*N)$ and hence we get
$$
\mathrm{rank}(\widetilde{\phi}_*N)N^3-\mathrm{dim}(\widetilde{W})\mathrm{deg}(\widetilde{\phi}_*N)(N|_{\widetilde{\phi}^{-1}(p)})^2\ge 0,
$$
which is the desired inequality
since $\widetilde{\phi}_*N\simeq f_*\omega_f$.
\end{proof}

\begin{prfofthm1}
From \eqref{KfMeq}, Lemma~\ref{singlem} and Lemma~\ref{stablem}, we have
$$
K_f^2\ge M^2\ge \frac{d-1}{d-3}N^3\ge \frac{6(d-3)}{d-2}\chi_f.
$$
\qed \end{prfofthm1}

\begin{prop}[cf.\ \cite{kon}] \label{equalprop}
Let $f\colon S\to B$ be a relatively minimal plane curve fibration of degree $d\ge 4$.
Then the following are equivalent.

\smallskip

\noindent
$(\mathrm{i})$ $M^2=\cfrac{d-1}{d-3}N^3$.

\smallskip

\noindent
$(\mathrm{ii})$ $K_f^2=\cfrac{6(d-3)}{d-2}\chi_f$.

\smallskip

\noindent
$(\mathrm{iii})$ There exists a $\mathbb{P}^2$-bundle $\phi\colon W=\mathbb{P}(\mathcal{E})\to B$ and a member $X\in |d\mathcal{O}_W(1)+\phi^{*}\mathfrak{k}|$ with at most rational double points as singularities such that $S$ is the minimal resolution of $X$.
\end{prop}

\begin{proof}
We first show (iii) from (i).
From the proof of Lemma~\ref{singlem}, (i) implies that $p_a(C')=g(C)$ for general $C\in |M'+\widetilde{f}^{*}\mathfrak{a}|$, $T'^{g-1}E_{\tau}=0$ and $(Z+E)M'=0$.
The former implies that $X'$ has at most isolated singularities.
$T'^{g-1}E_{\tau}=0$ implies that $P=P'$ and $E_{\tau}=0$. 
Hence we have $M-M'=\widetilde{f}^{*}\mathfrak{c}$ and then $(Z+E)M=0$.
It follows from the nefness of $M$ that $ZM=0$ and $EM=0$.
Therefore, we have $Z^2=MZ+Z^2=\rho^{*}K_fZ\ge 0$.
Thus, by the Hodge index theorem, we get $Z=0$.
On the other hand, we have $\mathrm{deg}f_*\mathcal{L}+\beta''=0$ from the proof of Lemma~\ref{singlem} and the assumption (i).
Thus $X=X'$ in $W= W'$ is linearly equivalent to
$d\mathcal{O}_W(1)-\phi^{*}\mathfrak{d}$ for some divisor $\mathfrak{d}$ of degree $\mathrm{deg}f_*\mathcal{L}$.
It follows that

\begin{align*}
\chi(\mathcal{O}_X)&=\chi(\mathcal{O}_W)-\chi(\mathcal{O}_W(-X)) \\
&=1-b+\chi(\mathcal{O}_W(K_W+X)) \\
&=1-b+\chi(\mathrm{Sym}^{d-3}f_*\mathcal{L}\otimes (\mathrm{det}f_*\mathcal{L}\otimes \omega_B\otimes \mathcal{O}_B(-\mathfrak{d}))) \\
&=1-b+g(1-b)+\chi_f-c+g(2b-2) \\
&=\chi(\mathcal{O}_S)-c \\
&\le \chi(\mathcal{O}_S).
\end{align*}
On the other hand, since $\Phi\colon \widetilde{S}\to X$ is a resolution of singularities of $X$, we have $\chi(\mathcal{O}_X)\ge \chi(\mathcal{O}_{\widetilde{S}})=\chi(\mathcal{O}_S)$.
Hence $c=0$ and $X$ has at most rational singularities.
Since $X$ is a hypersurface of $W$, any singularity of $X$ is a rational double point.
We can see that $\widetilde{S}=S$ and $\mathfrak{d}=\mathrm{det}f_*\mathcal{L}$.


Next we show that (iii) implies (ii).
By a simple computation, we have
$$
K_f^2=d(d-1)(d-3)\mathrm{deg}\mathcal{E}+3(d-1)(d-3)k,
$$
where $k:=\mathrm{deg}\mathfrak{k}$.
Moreover, by the similar computation as above, we have

\begin{align*}
\chi(\mathcal{O}_X)&=\chi(\mathcal{O}_W)-\chi(\mathcal{O}_W(-X)) \\
&=(g-1)(b-1)+\frac{d(d-1)(d-2)}{6}\mathrm{deg}\mathcal{E}+\frac{(d-1)(d-2)}{2}k.
\end{align*}
Since $X$ has at most rational double points, we get 

\begin{align*}
\chi_f&=\chi(\mathcal{O}_S)-(g-1)(b-1)\\
&=\chi(\mathcal{O}_X)-(g-1)(b-1)\\
&=\frac{d(d-1)(d-2)}{6}\mathrm{deg}\mathcal{E}+\frac{(d-1)(d-2)}{2}k.
\end{align*}
Hence (ii) holds.


It is clear that (i) follows from (ii).
\end{proof}

We remark that any member $X\in |d\mathcal{O}_W(1)+\phi^{*}\mathfrak{k}|$
on the $\mathbb{P}^2$-bundle $\phi\colon W=\mathbb{P}_B(\mathcal{E})\to B$ satisfies $K_{\phi|_X}^2=\lambda_d\chi_{\phi|_X}$, if we put
$K_{\phi|_X}=K_X-(\phi|_X)^{*}K_B$ and $\chi_{\phi|_X}=\chi(\mathcal{O}_X)-(g-1)(b-1)$.
In particular, one sees immediately that the slope inequality in Theorem~\ref{lowerboundthm} is sharp, 
because any general $X$ as above is smooth and irreducible provided that $\mathfrak{k}$ is sufficiently ample.

\section{Algebraization of fibers}
We consider a proper surjective holomorphic map $f\colon S\to \Delta$ from a non-singular complex surface $S$ to a small disk $\Delta\subset \mathbb{C}$ centered at the origin $0$ such that the general fiber $f^{-1}(t)$ over $t\neq 0$ is a non-singular curve of genus $g$ and put $F_0=f^{-1}(0)$.
The pair $(f,F_0)$ is called a {\em fiber germ of genus $g$},
which we sometimes denote it simply by $F_0$ if there is no fear of confusion.
A fiber germ $(f,F_0)$ is {\em relatively minimal} if $F_0$ contains no $(-1)$-curves.
In the sequel, we always assume that any fiber germ is relatively minimal.
Two relatively minimal fiber germs $(f\colon S\to \Delta,F_0)$ and $(f'\colon S'\to \Delta,F'_0)$ are {\em holomorphically equivalent} if there exist biholomorphic maps $\phi\colon S\to S'$ and $\psi\colon \Delta\to \Delta$ with $\psi(0)=0$ such that $f'\circ\phi=\psi\circ f$ after shrinking $\Delta$ if necessary.
Let $\mathcal{A}$ be a set of holomorphically equivalence classes of fiber germs of genus $g$
and $\nu\colon \mathcal{A}\to \Sigma$ a map from $\mathcal{A}$ to a set $\Sigma$.
The map $\nu$ is an {\em algebraic invariant} (cf.\ \cite{Te}) if for any fiber germ $(f\colon S\to \Delta,F_0)$ in $\mathcal{A}$, there exists a natural number $n$ such that for any fiber germ $(f'\colon S'\to \Delta,F'_0)$ in $\mathcal{A}$ which satisfies $S_n\simeq S'_n$
over $\mathrm{Spec}\mathbb{C}[t]/(t^n)$, we have $\nu(f,F_0)=\nu(f',F'_0)$, where $S_n:=S\times_{\Delta}\mathrm{Spec}\mathbb{C}[t]/(t^n)$.
For example, the map $\mu\colon \mathcal{A}\to \widehat{\Gamma}_g$ which sends a fiber germ $(f,F_0)$ to its topological monodromy $\mu_f$ is an algebraic invariant, where $\Gamma_g$ is the mapping class group of genus $g$ and $\widehat{\Gamma}_g$ is the set of its conjugacy classes.

Let $\mathcal{A}_d$ denote the set of holomorphically equivalence classes of fiber germs whose general fiber is a smooth plane curve of degree $d$.
The following is our main theorem:

\begin{thm}\label{mainthm}
There exists a non-negative algebraic invariant $\mathrm{Ind}_d\colon \mathcal{A}_d\to \frac{1}{d-2}\mathbb{Z}_{\ge 0}$ such that
for any relatively minimal plane curve fibration $f\colon S\to B$ of degree $d$, the value $\mathrm{Ind}_d(F)$ equals to $0$ for any general fiber $F$ of $f$ and 
$$
K_f^2=\frac{6(d-3)}{d-2}\chi_f+\sum_{p\in B}\mathrm{Ind}_d(F_p)
$$
holds.
\end{thm}

Now, we define the function $\mathrm{Ind}_d$.
Let $(f\colon S\to \Delta,F_0)$ be a fiber germ in $\mathcal{A}_d$.
Then, by Theorem~\ref{glueingthm}, there exists a line bundle $\mathcal{L}$ on $S$ such that the restriction $\mathcal{L}|_{F_t}$ is a $\mathfrak{g}^2_d$ on $F_t=f^{-1}(t)$ for any $t\neq 0$ and it is unique up to a multiple of a divisor consisting of components of $F_0=f^{-1}(0)$.
It follows that $\mathcal{L}^{\otimes d-3}(J)\simeq \omega_f$ for some divisor $J$ consisting of components of $F_0$.
Using the line bundle $\mathcal{L}$, we define $\mathrm{Ind}_d(F_0)$ by

\begin{align*}
\mathrm{Ind}_{d}(F_0):=&J^{2}+2(d-3)\left(\frac{d+1}{d-2}\mathrm{length}(R^{1}f_{*}\mathcal{L}^{\otimes d-2})-\mathrm{length}(R^{1}f_{*}\mathcal{L}^{\otimes d-1})\right)\\
&+2(d-3)\left(\mathrm{length}(\mathcal{T}_{d-1})-\frac{d+1}{d-2}\mathrm{length}(\mathcal{T}_{d-2})\right),
\end{align*}
where $\mathcal{T}_k$ is the torsion sheaf defined by the natural exact sequence
$$
0\to \mathrm{Sym}^{k}f_{*}\mathcal{L}\to f_{*}\mathcal{L}^{\otimes k}\to \mathcal{T}_k\to 0.
$$
We have seen that the value $\mathrm{Ind}_d(F_0)$ is independent of a choice of the line bundle $\mathcal{L}$ when the fiber germ $(f,F_0)$ is realized in a global fibration $S\to B$.
From \eqref{slopeeq}, in order to prove Theorem~\ref{mainthm}, we must show that for any fiber germ $(f,F_0)$ in $\mathcal{A}_d$, $\mathrm{Ind}_d(F_0)$ is well-defined, that is, not depend on a choice of $\mathcal{L}$ and non-negative algebraic invariant.
The following is a key lemma.

\begin{lem} \label{alglem}
For any fiber germ $(f\colon S\to \Delta,F_0)$ in $\mathcal{A}_d$ and any natural number $n$, there exists a plane curve fibration $\overline{f}\colon \overline{S}\to \mathbb{P}^{1}$ of degree $d$ such that $S_n$ is isomorphic to $\overline{S}_n:=\overline{S}\times_{\mathbb{P}^{1}}\mathrm{Spec}\mathcal{O}_{\mathbb{P}^{1},0}/\mathfrak{m}^{n}$ over $\mathrm{Spec}\mathbb{C}[t]/(t^n)\simeq \mathrm{Spec}\mathcal{O}_{\mathbb{P}^{1},0}/\mathfrak{m}^{n}$ and all the other singular fibers of $\overline{f}$ are irreducible Lefschetz plane curves of degree $d$, where $\mathfrak{m}$ denotes the maximal ideal of $\mathcal{O}_{\mathbb{P}^{1},0}$.
\end{lem}

\begin{proof}
We can take a line bundle $\mathcal{L}$ on $S$ such that $\mathcal{L}|_{F_t}$ is the $\mathfrak{g}^2_d$ on $F_t$ for any $t\neq 0$ from Theorem~\ref{glueingthm}.
Thus, we can take a rational map $S\dasharrow \Delta\times \mathbb{P}^2$ over $\Delta$ that embeds $F_t$ to $\mathbb{P}^2=\{t\}\times \mathbb{P}^2$ for any $t\neq 0$.
Let $\varphi(t;X,Y,Z)$ be a defining equation of $F_t\subset \mathbb{P}^2_{(X:Y:Z)}$ for $t\neq 0$, which is a homogeneous polynomial of degree $d$ with respect to $X,Y,Z$ and determined uniquely up to a multiple of a constant.
We may assume that $\varphi(t;X,Y,Z)$ is holomorphic in $t\neq 0$ after shrinking $\Delta$ if necessary.
By Riemann's extension theorem, $\varphi(t;X,Y,Z)$ is holomorphic at $t=0$.
Thus the image of a rational map $S\dasharrow \Delta\times \mathbb{P}^2$ can be written as $X:=\{(t,(X:Y:Z))\in \Delta\times \mathbb{P}^2|\varphi(t;X,Y,Z)=0\}$.
Let
$$
\varphi(t;X,Y,Z)=\varphi(0;X,Y,Z)+t\frac{d\varphi}{d t}(0;X,Y,Z)+\cdots+\frac{t^m}{m!}\frac{d^m\varphi}{d t^m}(0;X,Y,Z)+\cdots
$$
be the Taylor expansion near $0\in \Delta$ and define
$$
\varphi^{[n]}(t;X,Y,Z):=\varphi(0;X,Y,Z)+t\frac{d\varphi}{d t}(0;X,Y,Z)+\cdots+\frac{t^n}{n!}\frac{d^n\varphi}{d t^n}(0;X,Y,Z).
$$
Take a sufficiently large $m\gg n$ and general homogeneous polynomials $\psi_{n+1}(X,Y,Z)$, \ldots, $\psi_{m}(X,Y,Z)$ of degree $d$.
Let $\Phi(t_0,t_1;X,Y,Z)$ be the homogenization of the polynomial
$$
\varphi^{[n]}(t;X,Y,Z)+t^{n+1}\psi_{n+1}(t;X,Y,Z)+\cdots+t^{m}\psi_{m}(t;X,Y,Z)
$$
with respect to $t\in \mathbb{C}$
and put $\overline{X}:=\{((t_0:t_1),(X:Y:Z))\in \mathbb{P}^1\times \mathbb{P}^2|\Phi(t_0,t_1;X,Y,Z)=0\}$.
Taking a resolution of singularities of $\overline{X}$ and its relatively minimal model over $\mathbb{P}^1$, we get a plane curve fibration $\overline{f}\colon \overline{S}\to \mathbb{P}^1$ of degree $d$ such that 
$S_n$ is isomorphic to $\overline{S}_n$.
Since $\psi_{n+1}$, \ldots, $\psi_{m}$ are general, any singular fiber of $\overline{f}$ over $\mathbb{P}^1\setminus \{0\}$ is an irreducible Lefschetz plane curve of degree $d$ by Kuno's result \cite{Ku}.
\end{proof}

\begin{lem}\label{alginvlem}
$\mathrm{Ind}_d\colon \mathcal{A}_d\to \mathbb{Q}$ is a well-defined algebraic invariant.
\end{lem}

\begin{proof}
Fix a fiber germ $(f,F_0)$ of $\mathcal{A}_d$ arbitrarily and denote by $\mathrm{Ind}_d^{\mathcal{L}}(F_0)$ the value $\mathrm{Ind}_d(F_0)$ defined by using a line bundle $\mathcal{L}$ as above.
Note that the value $\mathrm{Ind}_d^{\mathcal{L}}(F_0)$ is completely determined by the restriction $\mathcal{L}_n:=\mathcal{L}|_{S_n}$ for a sufficiently large $n$ (depending on $(f,F_0)$).
From Lemma~\ref{alglem}, we can take a plane curve fibration $\overline{f}\colon \overline{S}\to \mathbb{P}^1$ of degree $d$ such that $S_n$ is isomorphic to $\overline{S}_n$. 
We will show that the line bundle $\mathcal{L}_n$ is the restriction of some line bundle $\overline{\mathcal{L}}$ on $\overline{S}$ to $\overline{S}_n$ via the isomorphism $S_n\simeq \overline{S}_n$.
Note that the topological monodromies of $(f,F_0)$ and $(\overline{f},\overline{F}_0)$ are the same and $F_0\simeq \overline{F}_0$.
Take a subvariety $\mathcal{U}$ of the Kuranishi space of the stable model $F'_0$ of $(f,F_0)$ parametrizing smooth plane curves of degree $d$ or its limit and consider the universal family $\mathcal{C}\to \mathcal{U}$.
Then the cyclic group $G=\mathbb{Z}_N$ acts on $\mathcal{C}$ and $\mathcal{U}$ equivariantly and the quotient fibration $\mathcal{C}/G\to \mathcal{U}/G$ contains the two fiber germs $(f,F_0)$ and $(\overline{f},\overline{F}_0)$, where the number $N$ is the minimal pseudo-period of the topological monodromy of $f$.
We may assume that $\mathcal{C}/G$ and $\mathcal{U}/G$ are normal by taking normalizations.
Applying Theorem~\ref{glueingthm} to $\mathcal{C}/G\to \mathcal{U}/G$, we obtain a divisorial sheaf $\mathbb{L}$ on $\mathcal{C}/G$ such that the restriction of $\mathbb{L}$ to any general fiber is a $\mathfrak{g}^2_d$.
We can write $\mathcal{L}\simeq \mathbb{L}|_{S}\otimes \mathcal{O}_S(D)$ for some divisor $D$ consisting of components of $F_0$ and then $\mathcal{L}_n\simeq \mathbb{L}|_{\overline{S}}\otimes \mathcal{O}_{\overline{S}}(D)|_{\overline{S}_n}$,
where $\mathbb{L}|_{\overline{S}}$ is a line bundle on $\overline{S}$ obtained by glueing the restriction of $\mathbb{L}$ to a neighborhood of the fiber $\overline{F}_0$ with a line bundle on $\overline{S}\setminus \overline{F}_0$ obtained by Theorem~\ref{glueingthm}.
The line bundle $\overline{\mathcal{L}}:=\mathbb{L}|_{\overline{S}}\otimes \mathcal{O}_{\overline{S}}(D)$ is the desired one.
Since $\mathrm{Ind}_d^{\mathcal{L}}(F_0)$ and $\mathrm{Ind}_d^{\overline{\mathcal{L}}}(\overline{F}_0)$ are determined by $\mathcal{L}_n$, we have 
 $\mathrm{Ind}^{\mathcal{L}}_d(F_0)=\mathrm{Ind}^{\overline{\mathcal{L}}}_d(\overline{F}_0)$.
Since $\mathrm{Ind}^{\overline{\mathcal{L}}}_d(\overline{F}_0)$
is independent of the choice of the line bundle, we see that $\mathrm{Ind}_d$ is well-defined.
In order to prove that $\mathrm{Ind}_d$ is an algebraic invariant, we apply the similar arguments as above to any fiber germ $(f'\colon S'\to \Delta,F'_0)$ in $\mathcal{A}_d$ with $S_n\simeq S'_n$.
Thus we have $\mathrm{Ind}_d(F_0)=\mathrm{Ind}_d(F'_0)$ for a sufficiently large $n$.
Such a number $n$ depends only on $(f,F_0)$ and $\mathbb{L}$.
Thus $\mathrm{Ind}_d$ is an algebraic invariant.
\end{proof}

\begin{defn}
A fiber germ $(f\colon S\to \Delta,F_0)$ in $\mathcal{A}_d$ is called a {\em Lefschetz fiber germ of type $0$} if $S\subset \Delta\times \mathbb{P}^2$ and $F_0=f^{-1}(0)$ is an irreducible Lefschetz plane curve of degree $d$.
\end{defn}

\begin{lem}\label{Ind0lem}
For any Lefschetz fiber germ $(f,F_0)$ of type $0$ in $\mathcal{A}_d$, we have $\mathrm{Ind}_d(F_0)=0$.
\end{lem}

\begin{proof}
We can take a line bundle $\mathcal{L}$ defining $\mathrm{Ind}_d(F_0)$ such that $\mathcal{L}^{\otimes d-3}\simeq \omega_f$ by restricting $\mathcal{O}(1)$ on $\Delta\times \mathbb{P}^2$ to $S$.
Moreover, we can see that $R^{1}f_{*}\mathcal{L}^{\otimes d-2}=R^{1}f_{*}\mathcal{L}^{\otimes d-1}=\mathcal{T}_{d-2}=\mathcal{T}_{d-1}=0$ since $F_0$ is irreducible and $H^1(F_0,\mathcal{L}^{\otimes d-2}|_{F_0})=H^1(F_0,\mathcal{L}^{\otimes d-1}|_{F_0})=0$.
Thus we have $\mathrm{Ind}_d(F_0)=0$.
\end{proof}

\begin{lem}\label{poslem}
For any fiber germ $(f,F_0)$ in $\mathcal{A}_d$, the value $\mathrm{Ind}_d(F_0)$ is non-negative.
\end{lem}

\begin{proof}
Fix a fiber germ $(f\colon S\to \Delta,F_0)$ in $\mathcal{A}_d$ arbitrarily. 
Since $\mathrm{Ind}_d$ is an algebraic invariant,
we can take a natural number $n$ such that for any fiber germ $(f'\colon S'\to \Delta,F'_0)$ of $\mathcal{A}_d$ such that $S_n\simeq S'_n$, we have $\mathrm{Ind}_d(F_0)=\mathrm{Ind}_d(F'_0)$. 
From Lemma~\ref{alglem}, we can take a plane curve fibration $\overline{f}\colon \overline{S}\to \mathbb{P}^1$ of degree $d$ such that $S_n\simeq \overline{S}_n$ and any other fiber germ of $\overline{f}$ is Lefschetz of type $0$.
Thus we get from \eqref{slopeeq}, Theorem~\ref{lowerboundthm} and Lemma~\ref{Ind0lem} that
$$
\mathrm{Ind}_d(F_0)=\mathrm{Ind}_d(\overline{F}_0)=K_f^2-\frac{6(d-3)}{d-2}\chi_f\ge 0.
$$
\end{proof}

Combining \eqref{slopeeq} with Lemma~\ref{alginvlem} and Lemma~\ref{poslem}, we get Theorem~\ref{mainthm}.

\begin{prop}
For a fiber germ $(f\colon S\to \Delta,F_0)\in \mathcal{A}_d$, $\mathrm{Ind}_d(F_0)=0$ holds if and only if $S$ is obtained by resolving singularities of some family $X\subset \Delta\times \mathbb{P}^2$ of plane curves of degree $d$ with at most rational double points as singularities.
\end{prop}

\begin{proof}
From Proposition~\ref{equalprop} and Theorem~\ref{alglem}, we get the assertion.
\end{proof}

\section{Local signature}

For an oriented compact real $4$-dimensional manifold $X$, the {\em signature} $\mathrm{Sign}(X)$ is defined to be the number of positive eigenvalues minus the number of negative eigenvalues of the intersection form on $H^2(X)$.
For a given condition $\mathcal{P}$ on smooth curves, let $\mathcal{A}_{\mathcal{P}}$ be the set of holomorphically equivalence classes of fiber germs whose general fiber has the condition $\mathcal{P}$.
Then a $\mathbb{Q}$-valued function $\sigma\colon \mathcal{A}_{\mathcal{P}}\to \mathbb{Q}$
is a {\em local signature} 
if for any relatively minimal fibered surface $f\colon S\to B$ whose general fiber $F$ satisfies the condition $\mathcal{P}$, we have
$\sigma(F)=0$ and $\mathrm{Sign}(X)=\sum_{p\in B}\sigma(F_p)$.

First, we briefly review the study of local signatures from the topological point of view. For more details, \cite{Ku2} is a good survey.
Let $P=\Sigma_{0,3}$ denote a pair of pants, that is, an oriented real surface obtained from $\Sigma_0=S^{2}$ by removing $3$ open disks with embedded disjoint closures and fix a base point $p_0\in P\setminus \partial P$ and two based loops $l_1$, $l_2$, $l_3$ with the relation $l_1l_2l_3=1$ in $\pi_1(P, p_0)$ which is homotopic to one of boundaries of $P$ with the counter clockwise orientation, respectively.
Let $\Gamma_g$ denote the mapping class group of $\Sigma_g$, a closed oriented real surface of genus $g$.
It was shown by Meyer that there is a $1$-cocycle $\tau_g\colon \Gamma_g\times \Gamma_g\to \mathbb{Z}$, which is called {\em Meyer's signature cocycle} such that
$\tau_g(\alpha_1,\alpha_2)=\mathrm{Sign}(E(\alpha_1,\alpha_2))$, where $E(\alpha_1, \alpha_2)\to P$ is a $\Sigma_g$-bundle whose monodromy $\mu\colon \pi_1(P,p_0)=\langle l_1,l_2\rangle \to \Gamma_g$ sends $l_i$ to $\alpha_i$ for $i=1,2$ (such a bundle exists and is unique up to homeomorphism).

Let $\mathcal{P}$ be a property of smooth projective curves of genus $g$.
We consider the following condition:

\smallskip

\noindent
$(*)_{\mathcal{P}}$ There exist a group $\Gamma_{\mathcal{P}}$ and a homomorphism $\iota_{\mathcal{P}}\colon \Gamma_{\mathcal{P}}\to \Gamma_g$ such that
for any $1$-parameter family $f\colon X\to C$ of smooth projective curves with $\mathcal{P}$, the monodromy map $\mu_f\colon \pi_1(C,c_0)\to \Gamma_g$ factors through $\Gamma_{\mathcal{P}}$, that is, there is a homomorphism $\mu_{f,\mathcal{P}}\colon \pi_1(C,c_0)\to \Gamma_{\mathcal{P}}$ with $\mu_f=\iota_{\mathcal{P}}\circ \mu_{f,\mathcal{P}}$
and for any open analytic subset $C'\subset C$, the homomorphisms $\mu_{f,\mathcal{P}}$ and $\mu_{f',\mathcal{P}}$ are compatible with the natural homomorphism $\pi_1(C',c_0)\to \pi_1(C,c_0)$, where $c_0\in C'\subset C$ is a base point and $f':=f|_{f^{-1}(C')}$.
Moreover, the pull back $\iota_{\mathcal{P}}^{*}\tau_g$ is a $\mathbb{Q}$-coboundary, that is, there is a function $\phi_{\mathcal{P}}\colon \Gamma_{\mathcal{P}}\to \mathbb{Q}$ such that 
$$
\tau_g(\iota_{\mathcal{P}}(a_1), \iota_{\mathcal{P}}(a_2))=\phi_{\mathcal{P}}(a_1)+\phi_{\mathcal{P}}(a_2)-\phi_{\mathcal{P}}(a_1a_2)
$$
for any $a_1,a_2\in \Gamma_{\mathcal{P}}$.

\smallskip

\noindent
If $\mathcal{P}$ satisfies the condition $(*)_{\mathcal{P}}$, we call 
the function $\phi_{\mathcal{P}}$ the {\em Meyer function on $\Gamma_{\mathcal{P}}$}.
Under the above situation, we can define a local signature from the Meyer function:

\begin{prop}
Let $\mathcal{P}$ be a property of smooth projective curves of genus $g$ which satisfies $(*)_{\mathcal{P}}$ and
define the function $\sigma_{\mathcal{P}}\colon \mathcal{A}_{\mathcal{P}}\to \mathbb{Q}$ by
$$
\sigma_{\mathcal{P}}(F_0):=\mathrm{Sign}(f^{-1}(\overline{\Delta'}))+\phi_{\mathcal{P}}(\mu_{f_0, \mathcal{P}}(\partial \Delta'))
$$ for any fiber germ $(f\colon S\to \Delta,F_0=f^{-1}(0))$ with $\mathcal{P}$, where $\Delta' \subset \Delta$ is an open disk centered at $0$ with $\overline{\Delta'}\subset \Delta$ and $f_0:=f|_{f^{-1}(\Delta\setminus \{0\})}$.
Then $\sigma_{\mathcal{P}}$ defines a local signature, that is,
for any fibered surface $f\colon S\to B$ whose general fiber satisfies $\mathcal{P}$, we have
$\sigma_{\mathcal{P}}(F)=0$ for any general fiber $F$ of $f$ and
$$
\mathrm{Sign}(S)=\sum_{p\in B}\sigma_{\mathcal{P}}(F_p).
$$
\end{prop}

\begin{proof}
Let $f\colon S\to B$ be a fibered surface whose general fiber $F$ has the property $\mathcal{P}$ and the set $\{p_1,\ldots,p_N\}$ of points of $B$ such that the restriction $f_0=f|_{S_0}\colon S_0=f^{-1}(B_0)\to B_0=B\setminus\{p_1,\ldots,p_N\}$ is a family of smooth projective curves with $\mathcal{P}$.
Let $D_i$ be a small open disk neighborhood of $p_i$ in $B$.
We take a pants decomposition $B\setminus\cup_{i=1}^{N}D_i=\cup_{j}P_j$, homeomorphisms $P_j\simeq P=\Sigma_{0,3}$ and loops $l_{j,1}$, $l_{j,2}$ in $P_j$ which sends to $l_1$, $l_2$ via this homeomorphism, respectively.
Thus we have $f^{-1}(P_j)\simeq E(\mu_{f_0}(l_{j,1}),\mu_{f_0}(l_{j,2}))$ as $\Sigma_g$-bundles over $P_j$.
By the Novikov additivity, we have

\begin{align*}
\mathrm{Sign}(S)&=\sum_{i=1}^{N}\mathrm{Sign}(f^{-1}(\overline{D_i}))+\sum_{j}\mathrm{Sign}(f^{-1}(P_j)) \\
&=\sum_{i=1}^{N}\mathrm{Sign}(f^{-1}(\overline{D_i}))+\sum_{j}\mathrm{Sign}(E(\mu_{f_0}(l_{j,1}),\mu_{f_0}(l_{j,2}))) \\
&=\sum_{i=1}^{N}\mathrm{Sign}(f^{-1}(\overline{D_i}))+\sum_{j}\tau_g(\mu_{f_0}(l_{j,1}), \mu_{f_0}(l_{j,2})) \\
&=\sum_{i=1}^{N}\mathrm{Sign}(f^{-1}(\overline{D_i}))+\sum_{j}\iota_{\mathcal{P}}^{*}\tau_g(\mu_{f_0,\mathcal{P}}(l_{j,1}), \mu_{f_0,\mathcal{P}}(l_{j,2})) \\
&=\sum_{i=1}^{N}\mathrm{Sign}(f^{-1}(\overline{D_i}))+\sum_{j}\left(\phi_{\mathcal{P}}(\mu_{f_0,\mathcal{P}}(l_{j,1}))+\phi_{\mathcal{P}}(\mu_{f_0,\mathcal{P}}(l_{j,2}))-\phi_{\mathcal{P}}(\mu_{f_0,\mathcal{P}}(l_{j,1}l_{j,2}))\right) \\
&=\sum_{i=1}^{N}\mathrm{Sign}(f^{-1}(\overline{D_i}))+\sum_{i=1}^{N}\phi_{\mathcal{P}}(\mu_{f_0,\mathcal{P}}(\partial D_i))\\
&=\sum_{i=1}^{N}\sigma_{\mathcal{P}}(F_{p_i}).
\end{align*}
\end{proof}

\begin{exa}[\cite{En}]
Let $\mathcal{P}=\text{hyperelliptic curve of genus $g$}$.
Then the condition $\mathcal{P}$ satisfies $(*)_{\mathcal{P}}$.
In fact, $\Gamma_{\mathcal{P}}$ is the centralizer of the class of the hyperelliptic involution in $\Gamma_g$ and $\iota_{\mathcal{P}}$ is a natural injection $\Gamma_{\mathcal{P}}\subset \Gamma_g$.
\end{exa}

\begin{exa}[\cite{Ku}]
Let $\mathcal{P}=\text{plane curve of degree $d$}$.
Then the condition $\mathcal{P}$ satisfies $(*)_{\mathcal{P}}$.
We denote the corresponding local signature by $\sigma_d^{\mathrm{top}}$.
\end{exa}

On the other hand, we can define another local signature for plane curve fibrations by using the Horikawa index $\mathrm{Ind}_d$:

\begin{defn}
We define $\sigma_d^{\mathrm{alg}}\colon \mathcal{A}_d\to \mathbb{Q}$ by
$$
\sigma_d^{\mathrm{alg}}=\frac{4}{12-\lambda_d}\mathrm{Ind}_d-\frac{8-\lambda_d}{12-\lambda_d}e,
$$
where $\lambda_d:=6(d-3)/(d-2)$ and $e\colon \mathcal{A}_d\to \mathbb{Q}$ is defined by $e(f,F):=e_{\mathrm{top}}(F)-2+2g$, which is clearly an algebraic invariant.
\end{defn}

\begin{prop}[cf.\ \cite{ak}]
For a relatively minimal plane curve fibration $f\colon S\to B$ of degree $d$,
we have
$$
\mathrm{Sign}(S)=\sum_{p\in B}\sigma_d^{\mathrm{alg}}(F_p).
$$
\end{prop}

\begin{proof}
The claim holds from Hirzebruch's signature theorem $\mathrm{Sign}(S)=K_f^2-8\chi_f$, Theorem~\ref{mainthm} and Noether's formula $12\chi_f=K_f^2+e_f$.
\end{proof}

Now, we show that two local signatures $\sigma_d^{\mathrm{alg}}$ and $\sigma_d^{\mathrm{top}}$ coincide on $\mathcal{A}_d$:

\begin{thm}[cf.\ \cite{Te}]
We have $\sigma_d^{\mathrm{alg}}(F_0)=\sigma_d^{\mathrm{top}}(F_0)$ for any fiber germ $(f,F_0)$ in $\mathcal{A}_d$.
\end{thm}

\begin{proof}
We see that two functions $\sigma_d^{\mathrm{alg}}$ and $\sigma_d^{\mathrm{top}}$ are algebraic invariants.
Moreover, we have 
$$
\sigma_d^{\mathrm{alg}}(F_0)=\sigma_d^{\mathrm{top}}(F_0)=-\frac{d+1}{3(d-1)}
$$
for any Lefschetz fiber germ $(f,F_0)$ of type $0$.
Thus the claim holds from Lemma~\ref{alglem}.
\end{proof}

\section{Durfee-type inequality for hypersurface surface singularities}

In this section, we prove the following theorem as an application of Theorem~\ref{mainthm}:

\begin{thm} \label{cor}
Let $(X,0)$ be an isolated hypersurface surface singularity with Milnor number $\mu$ and geometric genus $p_g>0$.
Then we have
$$
6p_g\le \mu-\chi_{\mathrm{top}}(A),
$$
or equivalently,
$$
\sigma\le -2p_g-1-s,
$$
where $\chi_{\mathrm{top}}(A)$ is the topological Euler number of the exceptional set $A$ of the minimal resolution $\pi\colon \widetilde{X}\to X$
and $s$ is the number of irreducible components of $A$.
In particular, the strong conjecture holds if $\chi_{\mathrm{top}}(A)\ge 0$
and the week conjecture holds for any isolated hypersurface surface singularity.
\end{thm}

\begin{defn}
Let $(f\colon S\to \Delta,F_0)$ be a relatively minimal fiber germ of plane curves.
Then we can take a line bundle $\mathcal{L}$ on $S$ such that the restriction $\mathcal{L}|_{F}$ to the general fiber $F$ defines the embedding $F\subset \mathbb{P}^2$ from Theorem~\ref{glueingthm}.
Thus the relative linear system $f_{*}\mathcal{L}$ defines a birational map onto the image $S\dasharrow X\subset \Delta\times \mathbb{P}^2$.
If the image $X$ has only one isolated singularity $x$, we call the pair $(X,x)$ an {\em isolated hypersurface singularity associated to a fiber germ $f\colon S\to \Delta$ of plane curves}.
\end{defn}

For an isolated hypersurface singularity $(X,x)$ associated to a fiber germ $(f,F_0)$ of plane curves of degree $d$, the Horikawa index $\mathrm{Ind}_d(F_0)$ can be computed by some invariants of the singularity $(X,x)$:

\begin{lem}\label{Horlem}
Let $(X,x)$ be an isolated hypersurface singularity associated to a fiber germ $(f\colon S\to \Delta,F_0)$ of plane curves of degree $d$ with Milnor number $\mu$ and geometric genus $p_g$.
Then we have
$$
\mathrm{Ind}_d(F_0)=\mu-\left(6+\frac{6}{d-2}\right)p_g-\chi_{\mathrm{top}}(A)+1+\epsilon,
$$
where $\chi_{\mathrm{top}}(A)$ is the topological Euler number of the exceptional set $A$ of the minimal resolution of $(X,x)$ and $\epsilon$ is the number of blow-ups in the minimal desingularization of indeterminacy of the rational map $S\dasharrow X$.
\end{lem}

\begin{proof}
Let $(X,x)$ be an isolated hypersurface singularity associated to a fiber germ $(f\colon S\to \Delta,F_0)$ of plane curves of degree $d$.
Let $\pi\colon \widetilde{S}\to X$ be the minimal desingularization of indeterminacy of the rational map $S\dasharrow X$, which is nothing but the minimal resolution of $(X,x)$.
Taking algebraization of the fiber germ $f\colon S\to \Delta$ in the sense of Lemma~\ref{alglem}, we may assume that $\Delta=\mathbb{P}^1$.
Let $\widetilde{f}\colon \widetilde{S}\to \mathbb{P}^1$ and $\overline{f}\colon X\to \mathbb{P}^1$ denote the natural fibrations.
Let $K$ be the canonical cycle of the minimal resolution of $(X,x)$.
Then we have
$$
p_g=\chi_{\overline{f}}-\chi_{\widetilde{f}}=\chi_{\overline{f}}-\chi_f,\quad -K^2=K_{\overline{f}}^2-K_{\widetilde{f}}^2=K_{\overline{f}}^2-K_{f}^2+\epsilon.
$$
On the other hand, we have
$$
K_f^2=\frac{6(d-3)}{d-2}\chi_f+\mathrm{Ind}_d(F_0),\quad K_{\overline{f}}^2=\frac{6(d-3)}{d-2}\chi_{\overline{f}},
$$
where the latter is obtained by a computation similar to that in the proof of Proposition~\ref{equalprop}.
Thus we get

\begin{align*}
\mathrm{Ind}_d(F_0)&=(K_f^2-K_{\overline{f}}^2)-\frac{6(d-3)}{d-2}(\chi_f-\chi_{\overline{f}}) \\
&=K^2+\epsilon+\frac{6(d-3)}{d-2}p_g.
\end{align*}
Combining it with Laufer's formula $\mu=12p_g+K^2+\chi_{\mathrm{top}}(A)-1$ \cite{La}, the desired equality holds.
\end{proof}

\begin{lem}\label{emblem}
Any isolated hypersurface surface singularity is holomorphically equivalent to some isolated hypersurface singularity $(X,x)$ associated to a fiber germ $(f\colon S\to \Delta,F_0)$ of plane curves with the birational morphism $S=\widetilde{S}\to X$ (i.e., $\epsilon=0$).
\end{lem}

\begin{proof}
Let $(X,0)$, $X=\{h(y,z_1,z_2)=0\}\subset \mathbb{C}^3$ be any isolated hypersurface surface singularity.
We may assume that the defining equation $h(y,z_1,z_2)$ is a polynomial.
Taking compactification $\overline{X}$ of $X$ in $\mathbb{P}^1\times \mathbb{P}^2$, the defining equation of $\overline{X}$ can be written by the homogenization $\overline{h}(Y_0,Y_1;Z_0,Z_1,Z_2)$ of $h(y,z_1,z_2)$.
Adding sufficiently higher terms to $h(y,z_1,z_2)$, we may assume that $0\in X\subset  \overline{X}$ is the unique singularity of $\overline{X}$ and the central fiber $\overline{X}\cap (0\times \mathbb{P}^2)=\{\overline{h}(1,0;Z_0,Z_1,Z_2)=0\}$ is irreducible and non-rational.
Thus the composite $\overline{f}=p\circ\pi \colon \overline{S}\to \mathbb{P}^1$ of the minimal resolution $\pi \colon \overline{S}\to \overline{X}$ and the projection $p\colon \overline{X}\to \mathbb{P}^1$ is a relatively minimal plane curve fibration.
Taking the fiber germ $f\colon S\to \Delta$ of $\overline{f}$ at the origin $0$, the assertion follows.
\end{proof}

\begin{prfofthm2}
Let $(X,x)$ be an isolated hypersurface surface singularity with Milnor number $\mu$ and geometric genus $p_g>0$.
Note that $(X,x)$ is not a rational double point.
From Lemma~\ref{emblem}, we may assume that $(X,x)$ is an isolated hypersurface singularity associated to a fiber germ $(f,F_0)$ of plane curves of degree $d$ with $\epsilon=0$.
From Lemma~\ref{Horlem} and the positivity of the Horikawa index, we have
$$
\mu-\left(6+\frac{6}{d-2}\right)p_g-\chi_{\mathrm{top}}(A)+1=\mathrm{Ind}_d(F_0)>0.
$$
Thus we have
$$
\mu-6p_g-\chi_{\mathrm{top}}(A)> \frac{6}{d-2}p_g-1> -1.
$$
Since the left hand side of the above inequality is an integer, we get
$$
\mu-6p_g-\chi_{\mathrm{top}}(A)\ge 0.
$$
\qed
\end{prfofthm2}

\end{document}